\documentclass[12pt,a4paper]{article}
\usepackage[utf8]{inputenc}
\usepackage[english]{babel}
\usepackage[OT1]{fontenc}
\usepackage{amsmath}
\usepackage{amsfonts}
\usepackage{amssymb}
\usepackage[dvipsnames]{xcolor}
\usepackage{hyperref}
\usepackage{cleveref}
\usepackage{tikz}
\usepackage{pgfplots}
\usepackage{float}
\floatstyle{boxed}
\restylefloat{figure}

\def\1{\mathbf 1}

\def\1{\bold 1}

\def\div{\mathrm{div}\,}

\usepackage{amsthm}
\theoremstyle{theorem}
\newtheorem{theorem}{Theorem}[section]
\newtheorem{proposition}[theorem]{Proposition}
\newtheorem{lemma}[theorem]{Lemma}

\newtheorem{remark}[theorem]{Remark}
\newtheorem{corollary}[theorem]{Corollary}

\theoremstyle{plain}
\newtoks\thehProclaim
\newtheorem*{Proclaim}{\the\thehProclaim}

\theoremstyle{definition}
\newtoks\thehDefinition
\newtheorem*{Definition}{\the\thehDefinition}

\numberwithin{equation}{section}

\providecommand{\keywords}[1]{\textbf{{Key words:}} #1}
\providecommand{\subjclass}[1]{\textit{{2020 Mathematics Subject Classification.}} #1}
\begin{document}
\title{Homogenization of the first initial-boundary value problem for periodic hyperbolic systems. Principal term of approximation
\footnote{\subjclass{Primary 35B27. Secondary 35L53.}}
}
\author{Yulia Meshkova \footnote{ E-mail: {\tt{juliavmeshke@yandex.ru}}. }}

\maketitle

\begin{abstract}
Let $\mathcal{O}\subset \mathbb{R}^d$ be a bounded domain of class $C^{1,1}$. In $ L_2(\mathcal{O};\mathbb{C}^n)$, we consider a matrix elliptic second order differential operator $A_{D,\varepsilon}$ with the Dirichlet boundary condition. Here $\varepsilon >0$ is a small parameter. The coefficients of the operator $A_{D,\varepsilon}$ are  periodic and depend on $\mathbf{x}/\varepsilon$. The principal terms of approximations for the operator cosine and sine functions are given in the $(H^2\rightarrow L_2)$- and $(H^1\rightarrow L_2)$-operator norms, respectively. The error estimates are of the precise order $O(\varepsilon)$ for a fixed time. The results in operator terms are derived from the quantitative homogenization estimate for approximation of the solution of the initial-boundary value problem for the equation $(\partial _t^2+A_{D,\varepsilon})\mathbf{u}_\varepsilon =\mathbf{F}$. 
\end{abstract}

\keywords{periodic differential operators, homogenization, convergence rates, hyperbolic systems.}

\section*{Introduction}
The paper is devoted to homogenization of periodic differential operators (DO's). More precisely, we are interested in the so-called operator error estimates, i.~e., in quantitative homogenization results, admitting formulation as estimates in the uniform operator topology.

For elliptic and parabolic problems, estimates of such type are very well studied, see, e.~g., books \cite[Chapter 14]{CDGr} and \cite{Sh}, the survey \cite{ZhPasUMN}, the papers \cite{BSu,Su07} and references therein. The hyperbolic problems and the non-stationary Schr\"odinger equation were considered in \cite{BSu08}. See also very resent results \cite{DSuNew,SuUMN}.
 
\subsection{The class of operators} 

Let $\Gamma\subset\mathbb{R}^d$ be a lattice. 
For a $\Gamma$-periodic function $\psi$ in $\mathbb{R}^d$, we denote $\psi ^\varepsilon (\mathbf{x}):=\psi (\mathbf{x}/\varepsilon)$, where $\varepsilon >0$.

Let $\mathcal{O}\subset\mathbb{R}^d$ be a bounded domain of class $C^{1,1}$. In $L_2(\mathcal{O};\mathbb{C}^n)$ we study a wide class of matrix strongly elliptic operators $A_{D,\varepsilon}$ given by the differential expression $b(\mathbf{D})^*g^\varepsilon (\mathbf{x})b(\mathbf{D})$, $\varepsilon >0$, with the Dirichlet boundary condition. Here $g$ is a Hermitian matrix-valued function
in $\mathbb{R}^d$ (of size $m\times m$), positive definite and periodic with respect to the lattice $\Gamma$. The operator $b(\mathbf{D})$ is an $(m\times n)$-matrix
first order DO with constant coefficients. It is assumed that $m>n$; the symbol of $b(\mathbf{D})$ has maximal rank. This condition ensures strong ellipticity of the operator $A_{D,\varepsilon}$.

The simplest example of the operator under consideration is the acoustics 
operator $-\mathrm{div}\,g^\varepsilon(\mathbf{x})\nabla$. The operator of elasticity theory also can be written in the required form. These and other examples are considered in
\cite{BSu} in detail.

Let $A_D^0$ be the effective operator $b(\mathbf{D})^*g^0b(\mathbf{D})$ defined on $H^2(\mathcal{O};\mathbb{C}^n)\cap H^1_0(\mathcal{O};\mathbb{C}^n)$. Here $g^0$ is the constant positive definite effective matrix.

\subsection{Known results in a bounded domain}

The broad literature is devoted to the operator error estimates in homogenization. In the present subsection, we concentrate on problems in a bounded domain.

Operator error estimates for the Dirichlet problems for second order elliptic equations in a bounded domain with sufficiently smooth boundary
were studied by many authors. Apparently, the first result is due to Sh. Moskow and M. Vogelius who proved an estimate
\begin{equation}
\label{Resolvent L2}
\Vert A_{D,\varepsilon}^{-1}-(A_D^0)^{-1}\Vert _{L_2(\mathcal{O})\rightarrow L_2(\mathcal{O})}\leqslant C\varepsilon,
\end{equation}
see \cite[Corollary 2.2]{MoV}. Here the operator $A_{D,\varepsilon}$ acts in $L_2(\mathcal{O})$, where $\mathcal{O}\subset\mathbb{R}^2$, and is given by $-\div g^\varepsilon (\mathbf{x}) \nabla$ with 
the Dirichlet condition on $\partial\mathcal{O}$. The matrix-valued function $g$ is assumed to be infinitely smooth.

For arbitrary dimension, homogenization problems in a bounded domain with sufficiently
smooth boundary were studied in \cite{Zh1, Zh2}, and \cite{ZhPas1}. The acoustics and elasticity operators with the Dirichlet or Neumann boundary conditions and without any smoothness assumptions on coefficients were considered.  The analog
of estimate \eqref{Resolvent L2}, but of order $O(\sqrt{\varepsilon})$, was obtained.  (In the case of the Dirichlet problem for
the acoustics equation, the $(L_2\rightarrow L_2)$-estimate was improved in \cite{ZhPas1}, but the order was
not sharp.) Similar results for the operator $-\mathrm{div}g^\varepsilon(\mathbf{x})\nabla$ in a smooth bounded domain $\mathcal{O}\subset\mathbb{R}^d$  
with the Dirichlet or Neumann boundary conditions were obtained by G. Griso \cite{Gr1, Gr2} with
the help of the “unfolding” method. In \cite{Gr2}, sharp-order estimate \eqref{Resolvent L2} (for the same operator) was proven. For elliptic systems similar results were independently obtained in \cite{KeLiSh}
and in \cite{PSu, SuL2}. Further results and a detailed survey can be found in \cite{Su3, Su4}.

The $(L_2\rightarrow  L_2)$-approximation for the parabolic semigroup $e^{-tA_{D,\varepsilon}}$ was proven in \cite{MSu_parabol}.

The first initial-boundary value problem for the hyperbolic systems were studied in \cite{M1}. By using the inverse Laplace transformation and a known result on homogenization of the resolvent in dependence on the spectral parameter, for the operator including the lower order terms it was obtained that
\begin{align}
\label{old cos}
&\left\Vert\left( \cos(tA_{D,\varepsilon}^{1/2})-\cos (t(A_D^0)^{1/2})\right)(A_D^0)^{-2}\right\Vert _{L_2(\mathcal{O})\rightarrow L_2(\mathcal{O})}\leqslant C\varepsilon (1+\vert t\vert ^5),
\\
\label{old sin}
\begin{split}
&\left\Vert\left( A_{D,\varepsilon}^{-1/2}\sin (tA_{D,\varepsilon}^{1/2})-(A_D^0)^{-1/2}\sin (t(A_D^0)^{1/2})\right)(A_D^0)^{-2}\right\Vert _{L_2(\mathcal{O})\rightarrow L_2(\mathcal{O})}
\\
&\leqslant C\varepsilon \vert t\vert (1+\vert t\vert ^5).
\end{split}
\end{align}
According to the known results in the whole space $\mathbb{R}^d
$, see \cite{BSu08,M2,DSu1}, these estimates do not look optimal with respect to the type of the norm and to the rate of growth with respect to the time $t$.

\subsection{Main results}

The main result of the paper is the improvement of estimates \eqref{old cos}, \eqref{old sin} with respect to the time growth and to the type of the operator norm: for $0<\varepsilon\leqslant 1$ and $t\in\mathbb{R}$,
\begin{align}
\label{intr cos}
&\left\Vert  \cos(tA_{D,\varepsilon}^{1/2})-\cos (t(A_D^0)^{1/2})\right\Vert _{H^2(\mathcal{O})\cap H^1_0(\mathcal{O})\rightarrow L_2(\mathcal{O})}\leqslant C\varepsilon (1+\vert t\vert),
\\
\label{intr sin}
&\left\Vert A_{D,\varepsilon}^{-1/2}\sin (tA_{D,\varepsilon}^{1/2})-(A_D^0)^{-1/2}\sin (t(A_D^0)^{1/2})\right\Vert _{H^1_0(\mathcal{O})\rightarrow L_2(\mathcal{O})}\leqslant C\varepsilon (1+\vert t\vert).
\end{align}
Here the space $H^2(\mathcal{O};\mathbb{C}^n)\cap H^1_0(\mathcal{O};\mathbb{C}^n)$ is equipped with the $H^2$-norm. 
For the operators, acting in $\mathbb{R}^d$, in \cite{DSu1} it was shown that estimates of the form \eqref{intr cos}, \eqref{intr sin} are optimal with respect to time $t$ and to the type of the operator norm in the general case. While in $\mathbb{R}^d$ it is possible to refine such estimates with respect to the type of the operator norm under additional assumptions on the operator (see \cite{DSu1}), for the problems in a bounded domain such a~refinement was not obtained in the present paper. 

Note that the operators $(A_D^0)^{-1}:L_2(\mathcal{O};\mathbb{C}^n)\rightarrow H^2(\mathcal{O};\mathbb{C}^n)\cap H^1_0(\mathcal{O};\mathbb{C}^n)$ and $(A_D^0)^{-1/2}:L_2(\mathcal{O};\mathbb{C}^n)\rightarrow H^1_0(\mathcal{O};\mathbb{C}^n)$ are isomorphisms, so estimates \eqref{intr cos}, \eqref{intr sin} can be reformulated as
\begin{align*}
&\left\Vert \left( \cos(tA_{D,\varepsilon}^{1/2})-\cos (t(A_D^0)^{1/2})\right)(A_D^0)^{-1}\right\Vert _{L_2(\mathcal{O})\rightarrow L_2(\mathcal{O})}\leqslant C\varepsilon (1+\vert t\vert),
\\
\begin{split}
&\left\Vert\left( A_{D,\varepsilon}^{-1/2}\sin (tA_{D,\varepsilon}^{1/2})-(A_D^0)^{-1/2}\sin (t(A_D^0)^{1/2})\right)(A_D^0)^{-1/2}\right\Vert _{L_2(\mathcal{O})\rightarrow L_2(\mathcal{O})}\\&\leqslant C\varepsilon (1+\vert t\vert).
\end{split}
\end{align*}

The $L_2$-operator error estimate for 
 homogenization of the solution to the first initial-boundary value problem for the hyperbolic system is also obtained.

\subsection{Method} The proof is a modification of the method of \cite{PSu,SuL2}. Consider solution $\mathbf{u}_\varepsilon$ of the first initial-boundary value problem for the hyperbolic equation $(\partial _t^2+A_\varepsilon)\mathbf{u}_\varepsilon =\mathbf{F}$ and the solution $\mathbf{u}_0$ of the corresponding effective problem. Introduce the first order approximation $\mathbf{v}_\varepsilon =\mathbf{u}_0+\varepsilon K(\varepsilon)\mathbf{u}_0$ to the solution, where the term $K(\varepsilon)\mathbf{u}_0$ is the corrector, and $\Vert \varepsilon K(\varepsilon)\mathbf{u}_0\Vert _{L_2(\mathcal{O})}=O(\varepsilon)$. The function $K(\varepsilon)\mathbf{u}_0$ does not satisfy the Dirichlet boundary condition, so we consider the corresponding boundary layer discrepancy $\mathbf{w}_\varepsilon$ and  estimate the difference $\mathbf{w}_\varepsilon-\varepsilon K(\varepsilon)\mathbf{u}_0$ in $L_2$. This estimation crucially relies on the $L_2$-boundedness of the operator $A_{D,\varepsilon}^{-1}A_\varepsilon$. To estimate $\Vert \mathbf{u}_\varepsilon -\mathbf{v}_\varepsilon +\mathbf{w}_\varepsilon\Vert_{L_2(\mathcal{O})}$ we use the approximation
\begin{equation}
\label{aux intr}
\Vert A_\varepsilon (I+\varepsilon K(\varepsilon))-A^0\Vert _{H^2(\mathcal{O})\cap H^1_0(\mathcal{O})\rightarrow\ H^{-1}(\mathcal{O})}\leqslant C\varepsilon ,\quad0<\varepsilon\leqslant 1,
\end{equation}
which is a direct consequence of \cite[Lemma~7.3]{PSu}.

\subsection{Plan of the paper} The paper consists of two sections and Introduction. In Section~\ref{Section Class of Operators}, we define the class of operators $A_{D,\varepsilon}$, introduce the effective operator $A_D^0$, and formulate the known auxiliary result \eqref{aux intr}. In Section~\ref{Section Hyperbolic systems}, we formulate and prove the main results of the paper. 

\subsection{Notation} Let $\mathfrak{H}$ and $\mathfrak{H}_\bullet$ be complex separable Hilbert spaces. The symbols $(\cdot ,\cdot)_\mathfrak{H}$ and $\Vert \cdot\Vert _\mathfrak{H}$ denote the inner product and the norm in   $\mathfrak{H}$, respectively; the symbol $\Vert \cdot\Vert _{\mathfrak{H}\rightarrow\mathfrak{H}_\bullet}$ means the norm of the linear continuous operators from $\mathfrak{H}$ to $\mathfrak{H}_\bullet$.

The symbols $\langle \cdot ,\cdot\rangle$ and $\vert \cdot\vert$ stand for the inner product and the norm in   $\mathbb{C}^n$, respectively, $\mathbf{1}_n$ is the identity $(n\times n)$-matrix. If $a$ is an~$(m\times n)$-matrix, then the symbol $\vert a\vert$ denotes the norm of the matrix $a$ as the operator from  $\mathbb{C}^n$ to $\mathbb{C}^m$. 

We use the notation $\mathbf{x}=(x_1,\dots , x_d)\in\mathbb{R}^d$, $iD_j=\partial _j =\partial /\partial x_j$, $j=1,\dots,d$, $\mathbf{D}=-i\nabla=(D_1,\dots ,D_d)$. The classes $L_p$ of vector-valued functions in a domain $\mathcal{O}\subset\mathbb{R}^d$ with values in $\mathbb{C}^n$ are denoted by $L_p(\mathcal{O};\mathbb{C}^n)$,\break $1\leqslant p\leqslant \infty$. The Sobolev spaces of $\mathbb{C}^n$-valued functions in a domain $\mathcal{O}\subset\mathbb{R}^d$ are denoted by $H^s(\mathcal{O};\mathbb{C}^n)$. 
For $n=1$, we simply write  $L_p(\mathcal{O})$, $H^s(\mathcal{O})$ and so on, but, sometimes, if this does not lead to confusion, we use such simple notation for the spaces of vector-valued or matrix-valued functions. The symbol $L_p((0,T);\mathfrak{H})$, $1\leqslant p\leqslant\infty$, denotes the $L_p$-space of $\mathfrak{H}$-valued functions on the interval $(0,T)$.

By $C$ and $c$ (possibly, with indices and marks) we denote various constants in estimates.

\section{Class of the operators}
\label{Section Class of Operators}

\subsection{Lattice in $\mathbb{R}^d$}
\label{Subsection lattices}
Let $\Gamma \subset \mathbb{R}^d$ be a lattice generated by a basis  $\mathbf{a}_1,\dots ,\mathbf{a}_d \in \mathbb{R}^d$, i.~e.,
$$
\Gamma =\left\lbrace
\mathbf{a}\in \mathbb{R}^d : \mathbf{a}=\sum _{j=1}^d \nu _j \mathbf{a}_j,\; \nu _j\in \mathbb{Z}
\right\rbrace ,
$$
and let $\Omega$ be the elementary cell of $\Gamma$:
$$
\Omega =
\left \lbrace
\mathbf{x}\in \mathbb{R}^d :\mathbf{x}=\sum _{j=1}^d \tau _j \mathbf{a}_j , \; -\frac{1}{2}<\tau _j<\frac{1}{2}
\right\rbrace .
$$
By $\vert \Omega \vert $ we denote the Lebesgue measure of $\Omega$: $\vert \Omega \vert =\mathrm{mea
s}\,\Omega$.

If $f(\mathbf{x})$ is a $\Gamma$-periodic function in $\mathbb{R}^d$, we denote
$$f^\varepsilon (\mathbf{x}):=f(\varepsilon ^{-1}\mathbf{x}),\quad \varepsilon >0.$$

\subsection{Operator $A_{D,\varepsilon}$}
\label{Subsection A_D,eps}
Let $\mathcal{O}\subset\mathbb{R}^d$ be a bounded domain of
class $C^{1,1}$. 
In $L_2(\mathcal{O};\mathbb{C}^n)$, we consider the operator $A_{D,\varepsilon}$ formally given by the differential expression  $A_\varepsilon =
b(\mathbf{D})^*g^\varepsilon (\mathbf{x})b(\mathbf{D})$ with the Dirichlet boundary condition. Here $g(\mathbf{x})$ is a $\Gamma$-periodic $(m\times m)$-matrix-valued function (in general, with complex entries). We assume that $g(\mathbf{x})>0$ and $g,g^{-1}\in L_\infty (\mathbb{R}^d)$. Next, $b(\mathbf{D})$ is the differential operator given  by
\begin{equation}
\label{b(D)=}
b(\mathbf{D})=\sum _{j=1}^d b_jD_j,
\end{equation}
where $b_j$, $j=1,\dots ,d$, are constant $(m\times n)$-matrices (in general, with complex entries). It is assumed that  $m\geqslant n$ and that the symbol $b(\boldsymbol{\xi})=\sum _{j=1}^d b_j\xi_j$ of the operator $b(\mathbf{D})$ has maximal rank:
$$
\mathrm{rank}\,b(\boldsymbol{\xi})=n,\quad 0\neq \boldsymbol{\xi}\in\mathbb{R}^d.
$$
This condition is equivalent to the estimates
\begin{equation}
\label{<b^*b<}
\alpha _0\mathbf{1}_n \leqslant b(\boldsymbol{\theta})^*b(\boldsymbol{\theta})
\leqslant \alpha _1\mathbf{1}_n,\quad
\boldsymbol{\theta}\in \mathbb{S}^{d-1},\quad
0<\alpha _0\leqslant \alpha _1<\infty,
\end{equation}
with some positive constants $\alpha _0$ and $\alpha _1$. So, 
\begin{equation}
\label{b_l<=}
\vert b_l\vert\leqslant \alpha _1^{1/2}.
\end{equation}

The precise definition of the operator $A_{D,\varepsilon}$ is given in terms of the quadratic form
\begin{equation*}
a_{D,\varepsilon}[\mathbf{u},\mathbf{u}]=\int _\mathcal{O}\langle g^\varepsilon b(\mathbf{D})\mathbf{u},b(\mathbf{D})\mathbf{u}\rangle\,d\mathbf{x},\quad \mathbf{u}\in H^1_0(\mathcal{O};\mathbb{C}^n).
\end{equation*}
This form is closed and positive definite. Indeed, extending
$\mathbf{u}$ by zero to $\mathbb{R}^d\setminus \mathcal{O}$, using the Fourier transformation and taking \eqref{<b^*b<} into
account, it is easy to check that
\begin{equation*}
\alpha _0\Vert g^{-1}\Vert _{L_\infty}\int _\mathcal{O}\vert \mathbf{D}\mathbf{u}\vert ^2\,d\mathbf{x}\leqslant a_{D,\varepsilon}[\mathbf{u},\mathbf{u}]
\leqslant \alpha _1\Vert g\Vert _{L_\infty}\int _\mathcal{O}\vert \mathbf{D}\mathbf{u}\vert ^2\,d\mathbf{x},
\end{equation*}
$\mathbf{u}\in H^1_0(\mathcal{O};\mathbb{C}^n)$. 
It remains to note that due to the Friedrichs's inequality the functional $\Vert \mathbf{D}\mathbf{u}\Vert _{L_2(\mathcal{O})}$  determines the norm in $H^1(\mathcal{O};\mathbb{C}^n)$ equivalent to the standard one. We have
\begin{equation}
\label{a_eps>=H1norm}
\begin{split}
\Vert A_{D,\varepsilon}^{1/2}\mathbf{u}\Vert _{L_2(\mathcal{O})}&\geqslant 2^{-1/2}(1+(\mathrm{diam}\,\mathcal{O})^{-2})^{1/2}\alpha _0^{1/2}\Vert g^{-1}\Vert _{L_\infty}^{1/2}\Vert \mathbf{u}\Vert _{H^1(\mathcal{O})}
\\
&=: c_*^{-1}\Vert \mathbf{u}\Vert _{H^1(\mathcal{O})},\quad \mathbf{u}\in H^1_0(\mathcal{O};\mathbb{C}^n).
\end{split}
\end{equation}
Hence,
\begin{equation}
\label{ADeps-1/2 L2 to H1}
\Vert A_{D,\varepsilon}^{-1/2}\Vert _{L_2(\mathcal{O})\rightarrow H^1(\mathcal{O})}
=
\Vert A_{D,\varepsilon}^{-1/2}\Vert _{H^{-1}(\mathcal{O})\rightarrow L_2(\mathcal{O})}
\leqslant c_*.
\end{equation}

\begin{lemma}
\label{Lemma ADeps -1 Aeps}
The operator $A_{D,\varepsilon}^{-1}A_\varepsilon$ is bounded in $L_2(\mathcal{O};\mathbb{C}^n)$ and 
\begin{equation}
\label{Lm ADeps -1 Aeps}
\Vert A_{D,\varepsilon}^{-1}A_\varepsilon \Vert _{L_2(\mathcal{O})\rightarrow L_2(\mathcal{O})}\leqslant 1.
\end{equation}
\end{lemma}

\begin{proof}
Let $\mathbf{f},\mathbf{h}\in C^\infty _0(\mathcal{O};\mathbb{C}^n)$. We have
\begin{equation*}
\begin{split}
&(A_{D,\varepsilon}^{-1}A_\varepsilon \mathbf{f},\mathbf{h})_{L_2(\mathcal{O})}=(A_\varepsilon\mathbf{f},A_{D,\varepsilon}^{-1}\mathbf{h})_{L_2(\mathcal{O})}
\\
&=\left( (g^\varepsilon)^{1/2}b(\mathbf{D})\mathbf{f},(g^\varepsilon)^{1/2}b(\mathbf{D})A_{D,\varepsilon}^{-1}\mathbf{h}\right))_{L_2(\mathcal{O})}
=(A_{D,\varepsilon}^{1/2}\mathbf{f},A_{D,\varepsilon}^{1/2}A_{D,\varepsilon}^{-1}\mathbf{h})_{L_2(\mathcal{O})}
\\
&=(A_{D,\varepsilon}^{1/2}\mathbf{f},A_{D,\varepsilon}^{-1/2}\mathbf{h})_{L_2(\mathcal{O})}=(\mathbf{f},\mathbf{h})_{L_2(\mathcal{O})}.
\end{split}
\end{equation*}
So,
\begin{equation*}
\vert (A_{D,\varepsilon}^{-1}A_\varepsilon \mathbf{f},\mathbf{h})_{L_2(\mathcal{O})}\vert \leqslant \Vert \mathbf{f}\Vert _{L_2(\mathcal{O})}\Vert \mathbf{h}\Vert _{L_2(\mathcal{O})}.
\end{equation*}
Since $\mathbf{h}$ belongs to the set $C_0^\infty(\mathcal{O};\mathbb{C}^n)$ which is dense in $L_2(\mathcal{O};\mathbb{C}^n)$, by continuity, 
\begin{equation*}
\Vert A_{D,\varepsilon}^{-1}A_\varepsilon \mathbf{f}\Vert _{L_2(\mathcal{O})}\leqslant \Vert \mathbf{f}\Vert _{L_2(\mathcal{O})}, \quad\mathbf{f}\in C_0^\infty (\mathcal{O};\mathbb{C}^n).
\end{equation*}
By continuity, this inequality is valid for any $\mathbf{f}\in L_2(\mathcal{O};\mathbb{C}^n)$. We arrive at estimate \eqref{Lm ADeps -1 Aeps}.
\end{proof}

\subsection{The effective operator}
\label{Subsection Effective operator}

Suppose that a $\Gamma$-periodic $(n\times m)$-matrix-valued function $\Lambda (\mathbf{x})$ is the (weak) solution of the problem
\begin{equation*}
b(\mathbf{D})^*g(\mathbf{x})(b(\mathbf{D})\Lambda (\mathbf{x})+\mathbf{1}_m)=0,\quad \int _{\Omega }\Lambda (\mathbf{x})\,d\mathbf{x}=0.
\end{equation*}
As $\Lambda$ is the weak solution, its $H^1(\Omega)$-norm is bounded. We will use estimate
\begin{align}
\label{Lambda <=}
\Vert \Lambda \Vert _{L_2(\Omega)}&\leqslant \vert \Omega \vert ^{1/2}M.
\end{align}
The constant $M$ can be written explicitly (see \cite[Subsection~7.3]{BSu05}) and depends only on $m$, $\alpha_0$, $\Vert g\Vert _{L_\infty}$, $\Vert g^{-1}\Vert _{L_\infty}$, and the parameters of the lattice~$\Gamma$.

The effective matrix is given by
\begin{equation*}
g^0=\vert \Omega \vert ^{-1}\int _{\Omega} g(\mathbf{x})(b(\mathbf{D})\Lambda (\mathbf{x})+\mathbf{1}_m)\,d\mathbf{x}.
\end{equation*}
It can be checked that $g^0$ is positive definite. Due to the Voight-Reuss bracketing (see, e.~g., \cite[Chapter~3, Theorem~1.5]{BSu}), the matrix $g^0$ satisfy estimates
\begin{equation}
\label{g0 estimates}
\vert g^0\vert \leqslant \Vert g\Vert _{L_\infty},\quad \vert (g^0)^{-1}\vert \leqslant\Vert g^{-1}\Vert _{L_\infty}.
\end{equation}

The effective operator $A_D^0$ for $A_{D,\varepsilon}$ is given by the differential expression 
\begin{equation}
\label{AD0}
A^0=b(\mathbf{D})^*g^0b(\mathbf{D})
\end{equation}
with the Dirichlet condition on $\partial\mathcal{O}$. The domain of this operator coincides with $ H^2(\mathcal{O};\mathbb{C}^n)\cap H^1_0(\mathcal{O};\mathbb{C}^n)$. Indeed, the operator \eqref{AD0} is strongly elliptic and 
due to the assumption $\partial\mathcal{O}\in C^{1,1}$, we can apply the ”additional smoothness”
theorems for solutions of strongly elliptic systems (see, e. g., \cite[Chapter
4]{McL}). Thus, $(A_D^0)^{-1}$ is a continuous operator from $L_2(\mathcal{O};\mathbb{C}^n)$ to $H^2(\mathcal{O};\mathbb{C}^n)$:
\begin{equation}
\label{AD0-1toH2}
\Vert (A_D^0)^{-1}\Vert _{L_2(\mathcal{O})\rightarrow H^2(\mathcal{O})}\leqslant \widehat{c}.
\end{equation}
The constant $\widehat{c}$ depends only on $\alpha _0$, $\alpha_1$, $\Vert g\Vert _{L_\infty}$, $\Vert g^{-1}\Vert _{L_\infty}$, the parameters of the lattice $\Gamma$, and the domain $\mathcal{O}$. 

\begin{remark}
Instead of the condition $\partial \mathcal{O}\in C^{1,1}$ one can impose the following implicit 
condition: a bounded domain $\mathcal{O}\subset\mathbb{R}^d$ with Lipschitz boundary is such that estimate \eqref{AD0-1toH2} 
holds. The results of the paper remain true for such domain. In the case of the scalar elliptic
operators, wide sufficient conditions on $\partial\mathcal{O}$ ensuring \eqref{AD0-1toH2} can be found in \textnormal{\cite{KoE}} and \textnormal{\cite[\textit{Chapter}~7]{MaSh}} \textnormal{(}in particular, it suffices that $\partial \mathcal{O}\in C^\alpha$ for $\alpha >3/2$\textnormal{)}.
\end{remark}

Similarly to \eqref{a_eps>=H1norm}, by \eqref{<b^*b<} and \eqref{g0 estimates},
\begin{equation*}
\Vert (A_D^0)^{1/2}\mathbf{u}\Vert _{L_2(\mathcal{O})}\geqslant c_*^{-1}\Vert \mathbf{u}\Vert _{H^1(\mathcal{O})},\quad\mathbf{u}\in H^1_0(\mathcal{O};\mathbb{C}^n).
\end{equation*}
Hence,
\begin{equation}
\label{1.10a}
\Vert (A_D^0)^{-1/2}\Vert _{L_2(\mathcal{O})\rightarrow H^1(\mathcal{O})}\leqslant c_*.
\end{equation}

\begin{lemma}
\label{Lemma AD0h}
For any $\mathbf{h}\in  H^2(\mathcal{O};\mathbb{C}^n)\cap H^1_0(\mathcal{O};\mathbb{C}^n)$ we have
\begin{equation*}
\Vert A_D^0\mathbf{h}\Vert _{L_2(\mathcal{O})}\leqslant \alpha _1d\Vert g\Vert _{L_\infty}\Vert \mathbf{D}^2 \mathbf{h}\Vert _{L_2(\mathcal{O})}.
\end{equation*}
\end{lemma}

\begin{proof}
By \eqref{b(D)=}, \eqref{b_l<=}, and \eqref{g0 estimates}, 
\begin{equation*}
\begin{split}
\Vert A_D^0\mathbf{h}\Vert _{L_2(\mathcal{O})}
&\leqslant \sum _{l=1}^d\Vert b_l^*D_lg^0 b(\mathbf{D})\mathbf{h}\Vert _{L_2(\mathcal{O})}
\leqslant \left(\sum _{l=1}^d \vert b_l\vert ^2\right)^{1/2}\Vert\mathbf{D}g^0b(\mathbf{D})\mathbf{h}\Vert _{L_2(\mathcal{O})}
\\
&\leqslant \alpha _1^{1/2}d^{1/2}\Vert g\Vert _{L_\infty}\Vert b(\mathbf{D})\mathbf{D}\mathbf{h}\Vert _{L_2(\mathcal{O})}
\leqslant \alpha _1d\Vert g\Vert _{L_\infty}\Vert \mathbf{D}^2\mathbf{h}\Vert _{L_2(\mathcal{O})}.
\end{split}
\end{equation*}
\end{proof}

\subsection{Steklov smoothing operator}

Let $S_\varepsilon$ be the Steklov smoothing (or Steklov averaging) operator \cite{St} in $L_2(\mathbb{R}^d;\mathbb{C}^m)$:
\begin{equation*}
(S_\varepsilon \mathbf{u})(\mathbf{x})=\vert \Omega\vert ^{-1}\int _\Omega \mathbf{u}(\mathbf{x}-\varepsilon\mathbf{z})\,d\mathbf{z}.
\end{equation*}

We need the following property of the operator $S_\varepsilon$ (see \cite{ZhPas1} or \cite[Proposition~3.2]{PSu}).

\begin{proposition}
\label{Proposition f Pi on H-kappa}
Let $f$ be a $\Gamma$-periodic function in $\mathbb{R}^d$ such that $f\in L_2(\Omega)$. Let $[f^\varepsilon]$ be the operator of multiplication by the function $f^\varepsilon(\mathbf{x})$.  Then
\begin{equation*}
\Vert [f^\varepsilon]S _\varepsilon \Vert _{L_2(\mathbb{R}^d)\rightarrow L_2(\mathbb{R}^d)}\leqslant\vert \Omega\vert ^{-1/2}\Vert f\Vert _{L_2(\Omega)},\quad \varepsilon >0.
\end{equation*}
\end{proposition}

\subsection{Auxiliary result}

We fix a linear continuous extension operator 
\begin{equation}
\label{P_O}
P_\mathcal{O}:H^l(\mathcal{O};\mathbb{C}^n)
\rightarrow H^l(\mathbb{R}^d;\mathbb{C}^n),\quad l=1,2
.
\end{equation}
Let 
\begin{equation}
\label{C_O}
C_\mathcal{O}:=\Vert P_\mathcal{O}\Vert _{H^1(\mathcal{O})\rightarrow H^1(\mathbb{R}^d)}.
\end{equation}
By $R_\mathcal{O}$ we denote the operator of restriction of functions in $\mathbb{R}^d$ onto the domain $\mathcal{O}$.

The following result was obtained in \cite[Lemma~7.3]{PSu}.

\begin{lemma}[\cite{PSu}]
Let $\boldsymbol{\Phi}\in L_2(\mathcal{O};\mathbb{C}^n)$. When for $0<\varepsilon\leqslant 1$ we have
\begin{equation}
\label{lm7.3PSu}
\begin{split}
\Vert& A_\varepsilon (I+\varepsilon R_\mathcal{O}[\Lambda ^\varepsilon ]S_\varepsilon b(\mathbf{D})P_\mathcal{O})(A_D^0)^{-1}\boldsymbol{\Phi}-A^0(A_D^0)^{-1}\boldsymbol{\Phi}\Vert _{H^{-1}(\mathcal{O})}
\\
&\leqslant C_1\varepsilon\Vert (A_D^0)^{-1}\boldsymbol{\Phi}\Vert _{H^2(\mathcal{O})}.
\end{split}
\end{equation}
The constant here depends only on $m$, $d$, $\alpha _0$, $\alpha _1$, $\Vert g\Vert _{L_\infty}$, $\Vert g^{-1}\Vert _{L_\infty}$,  the parameters of the lattice $\Gamma$, and the domain $\mathcal{O}$.
\end{lemma}

Since $(A_D^0)^{-1}:L_2(\mathcal{O};\mathbb{C}^n)\rightarrow H^2(\mathcal{O};\mathbb{C}^n)\cap H^1_0(\mathcal{O};\mathbb{C}^n)$ is an isomorphism, the following analogue of estimate \eqref{lm7.3PSu} holds.

\begin{corollary}
\label{Corollary Aeps()-A0}
Let $\mathbf{f}\in H^2(\mathcal{O};\mathbb{C}^n)\cap H^1_0(\mathcal{O};\mathbb{C}^n)$. When for $0<\varepsilon\leqslant 1$ we have
\begin{equation*}
\Vert A_\varepsilon (I+\varepsilon R_\mathcal{O}[\Lambda ^\varepsilon ]S_\varepsilon b(\mathbf{D})P_\mathcal{O})\mathbf{f}-A^0\mathbf{f}\Vert _{H^{-1}(\mathcal{O})}\leqslant C_1\varepsilon\Vert \mathbf{f}\Vert _{H^2(\mathcal{O})}.
\end{equation*}
\end{corollary}

\section{Hyperbolic systems. Main result}
\label{Section Hyperbolic systems}

\subsection{Problem setting}

\label{Subsection Problem setting}

Let $\mathbf{u}_\varepsilon$ be the solution of the first initial-boundary value problem for the hyperbolic system:
\begin{equation}
\label{u_eps problem F=0}
\begin{cases}
(\partial _t^2 +A_\varepsilon)\mathbf{u}_\varepsilon (\mathbf{x},t)=\mathbf{F}(\mathbf{x},t),\quad \mathbf{x}\in\mathcal{O},\;t\in (0,T),\\
\mathbf{u}_\varepsilon \vert _{\partial\mathcal{O}}=0,\\
\mathbf{u}_\varepsilon (\mathbf{x},0)=\boldsymbol{\phi}(\mathbf{x}),\quad (\partial _t\mathbf{u}_\varepsilon)(\mathbf{x},0)=\boldsymbol{\psi}(\mathbf{x}),\quad\mathbf{x}\in\mathcal{O}.
\end{cases}
\end{equation}
Here the initial data $\boldsymbol{\phi}\in H^2(\mathcal{O};\mathbb{C}^n)\cap H^1_0(\mathcal{O};\mathbb{C}^n)$, $\boldsymbol{\psi}\in H^1_0(\mathcal{O};\mathbb{C}^n)$, and the right-hand side $\mathbf{F}\in L_1((0,T);H^1_0(\mathcal{O};\mathbb{C}^n))$ for some $0<T\leqslant\infty$. When
\begin{equation}
\label{ueps=}
\begin{split}
\mathbf{u}_\varepsilon (\cdot ,t)&=\cos (tA_{D,\varepsilon}^{1/2})\boldsymbol{\phi}+A_{D,\varepsilon}^{-1/2}\sin (tA_{D,\varepsilon}^{1/2})\boldsymbol{\psi}
\\
&+\int _0^t A_{D,\varepsilon} ^{-1/2}\sin ((t-s)A_{D,\varepsilon} ^{1/2})\mathbf{F}(\cdot ,s)\,ds.
\end{split}
\end{equation}

\subsection{The effective problem}

\label{Subsection Effective problem}

The effective problem has the form
\begin{equation}
\label{u0 problem F=0}
\begin{cases}
(\partial _t^2 +A^0)\mathbf{u}_0 (\mathbf{x},t)=\mathbf{F}(\mathbf{x},t),\quad \mathbf{x}\in\mathcal{O},\;t\in (0,T),\\
\mathbf{u}_0 \vert _{\partial\mathcal{O}}=0,\\
\mathbf{u}_0 (\mathbf{x},0)=\boldsymbol{\phi}(\mathbf{x}),\quad (\partial _t\mathbf{u}_0)(\mathbf{x},0)=\boldsymbol{\psi}(\mathbf{x}),\quad\mathbf{x}\in\mathcal{O}.
\end{cases}
\end{equation}
We have
\begin{equation}
\label{u_0=}
\begin{split}
\mathbf{u}_0(\cdot ,t)&=\cos (t(A_D^0)^{1/2})\boldsymbol{\phi}+(A_D^0)^{-1/2}\sin (t(A_D^0)^{1/2})\boldsymbol{\psi}
\\
&+\int _0^t (A_{D}^0) ^{-1/2}\sin ((t-s)(A_{D}^0) ^{1/2})\mathbf{F}(\cdot ,s)\,ds.
\end{split}
\end{equation}
So, 
\begin{equation}
\label{u0 in H2 <= start}
\begin{split}
\Vert \mathbf{u}_0(\cdot ,t)\Vert _{H^2(\mathcal{O})}
&\leqslant\Vert \boldsymbol{\phi}\Vert _{H^2(\mathcal{O})}+\Vert (\mathbf{D}^2 +I)(A_D^0)^{-1/2}\boldsymbol{\psi}\Vert _{L_2(\mathcal{O})}
\\
&+\int _0^t\Vert (\mathbf{D}^2 +I)(A_D^0)^{-1/2}\mathbf{F}(\cdot ,s)\Vert _{L_2(\mathcal{O})}\,ds.
\end{split}
\end{equation}
By \eqref{1.10a}, 
\begin{equation*}
\begin{split}
\Vert (\mathbf{D}^2 +I)(A_D^0)^{-1/2}\boldsymbol{\psi}\Vert _{L_2(\mathcal{O})}
&\leqslant \Vert(\mathbf{D}^2 +I)^{1/2}(A_D^0)^{-1/2}\Vert _{L_2(\mathcal{O})\rightarrow L_2(\mathcal{O})}
\\
&\times\Vert (\mathbf{D}^2 +I)^{1/2}\boldsymbol{\psi}\Vert _{L_2(\mathcal{O})}
\leqslant c_*\Vert  \boldsymbol{\psi}\Vert _{H^1(\mathcal{O})}.
\end{split}
\end{equation*}
The summand with $\mathbf{F}$ in \eqref{u0 in H2 <= start} can be estimated in the same manner. Combining this with \eqref{u0 in H2 <= start}, we get
\begin{equation}
\label{u0 in H2}
\Vert \mathbf{u}_0(\cdot ,t)\Vert _{H^2(\mathcal{O})}\leqslant\Vert \boldsymbol{\phi}\Vert _{H^2(\mathcal{O})}+c_*\Vert\boldsymbol{\psi}\Vert _{H^1(\mathcal{O})}
+c_*\Vert\mathbf{F}\Vert_{L_1((0,t);H^1(\mathcal{O}))}.
\end{equation}

\subsection{Homogenization of solutions of the first initial-boundary value problem for hyperbolic systems}

Our \textit{main result} is the following theorem.

\begin{theorem}
\label{Theorem solutions}
Under the assumptions of Subsections~\textnormal{\ref{Subsection lattices}--\ref{Subsection Effective operator}} and \textnormal{\ref{Subsection Problem setting}, \ref{Subsection Effective problem}}, for $0<\varepsilon\leqslant 1$ and $t\in (0,T)$, we have
\begin{equation}
\label{Th solutions}
\begin{split}
\Vert &\mathbf{u}_\varepsilon (\cdot ,t)-\mathbf{u}_0(\cdot ,t)\Vert _{L_2(\mathcal{O})}
\\
&\leqslant C_2\varepsilon  (1+\vert t\vert )\left(\Vert \boldsymbol{\phi}\Vert _{H^2(\mathcal{O})}
+\Vert \boldsymbol{\psi}\Vert _{H^{1}(\mathcal{O})}+\Vert \mathbf{F}\Vert _{L_1((0,t);H^{1}(\mathcal{O}))}\right).
\end{split}
\end{equation}
The constant $C_2$ depends only on $m$, $d$, $\alpha _0$, $\alpha _1$, $\Vert g\Vert _{L_\infty}$, $\Vert g^{-1}\Vert _{L_\infty}$, the parameters of the lattice $\Gamma$, and the domain~$\mathcal{O}$.
\end{theorem}

\subsection{Main results in operator terms}

Since functions $\boldsymbol{\phi}\in H^2(\mathcal{O};\mathbb{C}^n)\cap H^1_0(\mathcal{O};\mathbb{C}^n)$ and $\boldsymbol{\psi}\in H^1_0(\mathcal{O};\mathbb{C}^n)$ in \eqref{ueps=} and \eqref{u_0=} are arbitrarily and one can take $\mathbf{F}=0$, Theorem~\ref{Theorem solutions} admits formulation in operator terms. 

\begin{theorem} 
\label{Theorem main result}
Let $\mathcal{O}$ be a bounded domain of class $C^{1,1}$. Let the assumptions of Subsections~\textnormal{\ref{Subsection lattices}--\ref{Subsection Effective operator}} be satisfied.  For $0<\varepsilon\leqslant 1$, $t\in\mathbb{R}$, we have
\begin{align*}
&\Vert \cos (tA_{D,\varepsilon}^{1/2})-\cos(t(A_D^0)^{1/2})\Vert _{H^2(\mathcal{O})\cap H^1_0(\mathcal{O})\rightarrow L_2(\mathcal{O})}\leqslant C_2\varepsilon (1+\vert t\vert),\\
&\Vert A_{D,\varepsilon}^{-1/2}\sin(tA_{D,\varepsilon}^{1/2})-(A_D^0)^{-1/2}\sin (t(A_D^0)^{1/2})\Vert _{H^1_0(\mathcal{O})\rightarrow L_2(\mathcal{O})}\leqslant C_2\varepsilon (1+\vert t\vert).
\end{align*}
Here the space $H^2(\mathcal{O};\mathbb{C}^n)\cap H^1_0(\mathcal{O};\mathbb{C}^n)$ is equipped with the $H^2$-norm. 
The constant $C_2$ depends only on $m$, $d$, $\alpha_0$, $\alpha _1$, $\Vert g\Vert _{L_\infty}$, $\Vert g^{-1}\Vert _{L_\infty}$, the parameters of the lattice $\Gamma$,  and the domain $\mathcal{O}$.
\end{theorem}

\subsection{Start of the proof of Theorem~\ref{Theorem solutions}. Discrepancy}

Let $\mathbf{v}_\varepsilon$ be the first order approximation to the solution $\mathbf{u}_\varepsilon$: 
\begin{equation}
\label{v_eps=}
\mathbf{v}_\varepsilon =\mathbf{u}_0+\varepsilon R_\mathcal{O}\Lambda ^\varepsilon S_\varepsilon  b(\mathbf{D})\widetilde{\mathbf{u}}_0.
\end{equation}
Here $\widetilde{\mathbf{u}}_0:=P_\mathcal{O}\mathbf{u}_0$, where $P_\mathcal{O}$ is the extension operator~\eqref{P_O}.

The function $\mathbf{u}_\varepsilon -\mathbf{v}_\varepsilon$ does not satisfy the Dirichlet boundary condition. It is convenient to introduce the discrepancy $\mathbf{w}_\varepsilon$ as the weak solution of the problem
\begin{equation}
\label{Discrepancy problem}
\begin{cases}
(\partial _t^2 +A_{\varepsilon})\mathbf{w}_\varepsilon =\varepsilon R_\mathcal{O}\Lambda ^\varepsilon S_\varepsilon b(\mathbf{D})\partial _t^2\widetilde{\mathbf{u}}_0\quad\mathrm{in }\;\mathcal{O},\\
\mathbf{w}_\varepsilon \vert _{\partial\mathcal{O}}=\varepsilon \Lambda ^\varepsilon S_\varepsilon  b(\mathbf{D})\widetilde{\mathbf{u}}_0\vert _{\partial\mathcal{O}},\\
\mathbf{w}_\varepsilon (\cdot ,0)=\varepsilon R_\mathcal{O}\Lambda ^\varepsilon S_\varepsilon  b(\mathbf{D})P_\mathcal{O}\boldsymbol{\phi},
\quad
(\partial _t\mathbf{w}_\varepsilon)(\cdot ,0)=\varepsilon R_\mathcal{O}\Lambda ^\varepsilon S_\varepsilon  b(\mathbf{D})P_\mathcal{O}\boldsymbol{\psi}.
\end{cases}
\end{equation}

Let us write \eqref{Discrepancy problem} as
\begin{equation*}
\begin{cases}
(\partial _t^2 +A_{\varepsilon})(\mathbf{w}_\varepsilon -\varepsilon R_\mathcal{O} \Lambda ^\varepsilon S_\varepsilon b(\mathbf{D})\widetilde{\mathbf{u}}_0) =-A_{\varepsilon}\varepsilon R_\mathcal{O} \Lambda ^\varepsilon S_\varepsilon b(\mathbf{D})\widetilde{\mathbf{u}}_0\quad\mathrm{in}\;\mathcal{O},
\\
(\mathbf{w}_\varepsilon -\varepsilon   \Lambda ^\varepsilon S_\varepsilon b(\mathbf{D})\widetilde{\mathbf{u}}_0) \vert _{\partial\mathcal{O}}=0,\\
(\mathbf{w}_\varepsilon -\varepsilon  R_\mathcal{O}\Lambda ^\varepsilon S_\varepsilon b(\mathbf{D})\widetilde{\mathbf{u}}_0) (\cdot ,0)=
0,
\quad
\partial _t(\mathbf{w}_\varepsilon -\varepsilon  R_\mathcal{O}\Lambda ^\varepsilon S_\varepsilon b(\mathbf{D})\widetilde{\mathbf{u}}_0)(\cdot ,0)=
0.
\end{cases}
\end{equation*}
So,
\begin{equation}
\label{w_eps-...=}
\begin{split}
(\mathbf{w}_\varepsilon &-\varepsilon R_\mathcal{O}\Lambda ^\varepsilon S_\varepsilon b(\mathbf{D})\widetilde{\mathbf{u}}_0)(\cdot ,t)=
\\
&-\int_0^t A_{D,\varepsilon} ^{-1/2}\sin((t-s)A_{D,\varepsilon}^{1/2})A_{\varepsilon}\varepsilon R_\mathcal{O} \Lambda ^\varepsilon S_\varepsilon  b(\mathbf{D})\widetilde{\mathbf{u}}_0(\cdot ,s)\,ds
=:\mathcal{I}(\varepsilon ,t).\end{split}
\end{equation}

\begin{lemma}
\label{Lemma w-esp - estimate}
Let $\mathbf{u}_0$ be the solution of effective problem~\eqref{u0 problem F=0}. Let the discrepancy $\mathbf{w}_\varepsilon$ be the solution of problem \eqref{Discrepancy problem}.  
For $0<\varepsilon\leqslant 1$ and $t\in(0,T)$ we have
\begin{equation}
\label{w-eps - estimate}
\begin{split}
\Vert & \mathbf{w}_\varepsilon (\cdot,t)-\varepsilon R_\mathcal{O} \Lambda ^\varepsilon S_\varepsilon b(\mathbf{D})\widetilde{\mathbf{u}}_0(\cdot ,t)\Vert _{L_2(\mathcal{O})}
\\
&\leqslant C_3\varepsilon (1+\vert t\vert )(\Vert \boldsymbol{\phi}\Vert _{H^2(\mathcal{O})}+\Vert \boldsymbol{\psi}\Vert _{H^1(\mathcal{O})}+\Vert \mathbf{F}\Vert _{L_1((0,t);H^1(\mathcal{O}))}).
\end{split}
\end{equation}
 The constant $C_3$ depends only on $m$, $d$, $\alpha _0$, $\alpha _1$, $\Vert g\Vert _{L_\infty}$, $\Vert g^{-1}\Vert _{L_\infty}$, the parameters of the lattice $\Gamma$,  and the domain~$\mathcal{O}$.
\end{lemma}

\begin{proof}

Integrating by parts, we have
\begin{equation*}
\begin{split}
\mathcal{I}(\varepsilon ,t)&=-\int _0^t \frac{d\cos((t-s)A_{D,\varepsilon}^{1/2})}{ds}A_{D,\varepsilon}^{-1}A_{\varepsilon}\varepsilon R_\mathcal{O}\Lambda ^\varepsilon S_\varepsilon  b(\mathbf{D})\widetilde{\mathbf{u}}_0(\cdot ,s)\,ds
\\
&=-A_{D,\varepsilon}^{-1}A_\varepsilon\varepsilon R_\mathcal{O}\Lambda^\varepsilon S_\varepsilon b(\mathbf{D})\widetilde{\mathbf{u}}_0(\cdot,t)
\\
&+\cos (tA_{D,\varepsilon}^{1/2})A_{D,\varepsilon}^{-1}A_\varepsilon \varepsilon R_\mathcal{O}\Lambda ^\varepsilon S_\varepsilon  b(\mathbf{D})\widetilde{\mathbf{u}}_0(\cdot ,0)
\\
&+\int _0^t \cos ((t-s)A_{D,\varepsilon}^{1/2})A_{D,\varepsilon}^{-1}A_\varepsilon \varepsilon R_\mathcal{O}\Lambda ^\varepsilon S _\varepsilon  b(\mathbf{D})\partial _s\widetilde{\mathbf{u}}_0(\cdot ,s)\,ds.
\end{split}
\end{equation*}
By Lemma~\ref{Lemma ADeps -1 Aeps} and \eqref{u0 problem F=0}, 
\begin{equation*}
\begin{split}
\Vert \mathcal{I}(\varepsilon ,t)\Vert _{L_2(\mathcal{O})}
&\leqslant \varepsilon\Vert \Lambda ^\varepsilon S_\varepsilon b(\mathbf{D})\widetilde{\mathbf{u}}_0(\cdot ,t)\Vert _{L_2(\mathcal{O})}
+\varepsilon \Vert \Lambda ^\varepsilon S_\varepsilon b(\mathbf{D})P_\mathcal{O}\boldsymbol{\phi}\Vert _{L_2(\mathcal{O})}
\\
&+\varepsilon\int _0^t\Vert \Lambda ^\varepsilon S_\varepsilon b(\mathbf{D})\partial _s\widetilde{\mathbf{u}}_0(\cdot,s)\Vert _{L_2(\mathcal{O})}\,ds.
\end{split}
\end{equation*} 
Using \eqref{<b^*b<}, \eqref{Lambda <=}, \eqref{C_O}, and Proposition~\ref{Proposition f Pi on H-kappa}, we derive that
\begin{equation}
\label{I1<= almost finish}
\begin{split}
\Vert \mathcal{I}(\varepsilon ,t)\Vert _{L_2(\mathcal{O})}
&\leqslant \varepsilon M\alpha _1^{1/2}C_\mathcal{O}\Vert \widetilde{\mathbf{u}}_0(\cdot,t)\Vert _{H^1(\mathcal{O})}
+\varepsilon M \alpha _1^{1/2}C_\mathcal{O}\Vert \boldsymbol{\phi}\Vert _{H^1(\mathcal{O})}
\\
&+\varepsilon M\alpha _1^{1/2}\int _0^t\Vert \mathbf{D}\partial _s\widetilde{\mathbf{u}}_0(\cdot ,s)\Vert _{L_2(\mathcal{O})}\,ds.
\end{split}
\end{equation}

By \eqref{b_l<=}, \eqref{g0 estimates}, \eqref{C_O}, and \eqref{u_0=}, 
\begin{equation}
\label{D dt tilde uo<=}
\begin{split}
\Vert &\mathbf{D}\partial _s\widetilde{\mathbf{u}}_0 (\cdot ,s)\Vert _{L_2(\mathbb{R}^d)}
\\
&=\Bigl\Vert \mathbf{D}P_\mathcal{O}\Bigl(-(A_D^0)^{1/2}\sin (s(A_D^0)^{1/2})\boldsymbol{\phi}+\cos (s (A_D^0)^{1/2})\boldsymbol{\psi}
\\
&+\int _0^t\cos((t-s)(A_D^0)^{1/2})\mathbf{F}(\cdot ,s)\,ds
\Bigr)\Bigr\Vert _{L_2(\mathbb{R}^d)}
\\
&\leqslant 
C_\mathcal{O}\left(\alpha _1^{1/2}d^{1/2}\Vert g\Vert ^{1/2}_{L_\infty}\Vert \boldsymbol{\phi}\Vert _{H^2(\mathcal{O})}+\Vert \boldsymbol{\psi}\Vert _{H^1(\mathcal{O})}
+\Vert \mathbf{F}\Vert _{L_1((0,t);H^1(\mathcal{O}))}\right).
\end{split}
\end{equation}
(We assumed that $\boldsymbol{\phi}\in H^2(\mathcal{O};\mathbb{C}^n)\cap H^1_0(\mathcal{O};\mathbb{C}^n)$, so we can apply the operator $(A_D^0)^{1/2}$ to $\boldsymbol{\phi}$.) 

Combining \eqref{u0 in H2}, \eqref{I1<= almost finish}, and \eqref{D dt tilde uo<=}, we get
\begin{equation}
\label{I1<=}
\Vert \mathcal{I}(\varepsilon ,t)\Vert _{L_2(\mathcal{O})}\leqslant C_3 \varepsilon (1+\vert t\vert )(\Vert \boldsymbol{\phi}\Vert _{H^2(\mathcal{O})}+\Vert \boldsymbol{\psi}\Vert _{H^1(\mathcal{O})}+\Vert \mathbf{F}\Vert _{L_1((0,t);H^1(\mathcal{O}))}),
\end{equation}
where $C_3 :=M\alpha _1^{1/2}C_\mathcal{O}\max\lbrace 2;c_*;\alpha _1^{1/2}d^{1/2}\Vert g\Vert _{L_\infty}\rbrace$. Taking \eqref{w_eps-...=} into account, we arrive at estimate \eqref{w-eps - estimate}.
\end{proof}

\subsection{End of the proof of Theorem~\ref{Theorem solutions}. $L_2$-estimate for the function $\mathbf{u}_\varepsilon -\mathbf{v}_\varepsilon+\mathbf{w}_\varepsilon$}
We have
\begin{equation*}
\begin{cases}
(\partial _t^2 +A_{\varepsilon})(\mathbf{u}_\varepsilon -\mathbf{v}_\varepsilon+\mathbf{w}_\varepsilon)=
A^0\mathbf{u}_0-A_\varepsilon \left(I+\varepsilon R_\mathcal{O}\Lambda ^\varepsilon S_\varepsilon b(\mathbf{D})P_\mathcal{O}\right){\mathbf{u}}_0\quad\mathrm{in}\;\mathcal{O},
\\
\mathbf{u}_\varepsilon -\mathbf{v}_\varepsilon+\mathbf{w}_\varepsilon\vert _{\partial\mathcal{O}}=0,\\
(\mathbf{u}_\varepsilon -\mathbf{v}_\varepsilon+\mathbf{w}_\varepsilon)(\cdot,0)=0,
\quad 
\partial _t(\mathbf{u}_\varepsilon -\mathbf{v}_\varepsilon+\mathbf{w}_\varepsilon)(\cdot,0)=0.
\end{cases}
\end{equation*}
So, 
\begin{equation*}
\begin{split}
(&\mathbf{u}_\varepsilon -\mathbf{v}_\varepsilon+\mathbf{w}_\varepsilon)(\cdot,t)
\\
&=
\int _0^t A_{D,\varepsilon}^{-1/2}\sin ((t-s)A_{D,\varepsilon}^{1/2})\left(A^0\mathbf{u}_0-A_\varepsilon \left(I+\varepsilon R_\mathcal{O} \Lambda ^\varepsilon S_\varepsilon b(\mathbf{D})P_\mathcal{O}\right){\mathbf{u}}_0\right)(\cdot,s)\,ds
.
\end{split}
\end{equation*}
By Corollary~\ref{Corollary Aeps()-A0},
\begin{equation*}
\Vert (\mathbf{u}_\varepsilon -\mathbf{v}_\varepsilon+\mathbf{w}_\varepsilon)(\cdot,t)\Vert _{L_2(\mathcal{O})}
\leqslant\int _0^t\Vert A_{D,\varepsilon}^{-1/2}\Vert _{H^{-1}(\mathcal{O})\rightarrow L_2(\mathcal{O})}C_1\varepsilon\Vert \mathbf{u}_0(\cdot,s)\Vert _{H^2(\mathcal{O})}\,ds.
\end{equation*}
(Our assumptions  on $\boldsymbol{\phi}$, $\boldsymbol{\psi}$, and $\mathbf{F}$ guarantee that $\mathbf{u}_0(\cdot,t)\in H^2(\mathcal{O};\mathbb{C}^n)\cap H^1_0(\mathcal{O};\mathbb{C}^n)$.) 
Together with \eqref{ADeps-1/2 L2 to H1} and \eqref{u0 in H2} this implies
\begin{equation*}
\begin{split}
\Vert &(\mathbf{u}_\varepsilon -\mathbf{v}_\varepsilon+\mathbf{w}_\varepsilon)(\cdot,t)\Vert _{L_2(\mathcal{O})}
\\
&\leqslant c_*C_1\varepsilon\vert t\vert 
(\Vert \boldsymbol{\phi}\Vert _{H^2(\mathcal{O})}+c_*\Vert \boldsymbol{\psi}\Vert _{H^1(\mathcal{O})}
+c_*\Vert\mathbf{F}\Vert _{L_1((0,t);H^1(\mathcal{O}))}
).
\end{split}
\end{equation*}

We arrive at the following result.

\begin{lemma}
\label{Lemma ueps-veps+weps}
Let $\mathbf{u}_\varepsilon$ be the solution of the problem~\eqref{u_eps problem F=0}. Let $\mathbf{v}_\varepsilon$ be the first order approximation \eqref{v_eps=} to the solution $\mathbf{u}_\varepsilon$, where $\mathbf{u}_0$ is the solution of the effective problem~\eqref{u0 problem F=0}. Let the discrepancy $\mathbf{w}_\varepsilon$ be the solution of the problem \eqref{Discrepancy problem}. 
When for $0<\varepsilon\leqslant 1$ and $t\in(0,T)$ we have
\begin{equation*}
\Vert (\mathbf{u}_\varepsilon -\mathbf{v}_\varepsilon +\mathbf{w}_\varepsilon )(\cdot,t)\Vert _{L_2(\mathcal{O})}\leqslant C_{4}\varepsilon\vert t\vert (\Vert \boldsymbol{\phi}\Vert _{H^2(\mathcal{O})}+\Vert \boldsymbol{\psi}\Vert _{H^1(\mathcal{O})}
+\Vert\mathbf{F}\Vert _{L_1((0,t);H^1(\mathcal{O}))}
).
\end{equation*}
The constant $C_{4}=c_*C_1\max\lbrace 1;c_*\rbrace $ depends only on $m$, $d$, $\alpha _0$, $\alpha _1$, $\Vert g\Vert _{L_\infty}$, $\Vert g^{-1}\Vert _{L_\infty}$, the parameters of the lattice $\Gamma$, 
and the domain $\mathcal{O}$.
\end{lemma}

Together with identity \eqref{v_eps=} and inequality \eqref{w-eps - estimate}, this result imply estimate 
\eqref{Th solutions} 
with the constant $C_2:=C_3+C_{4}$. The proof of Theorem~\ref{Theorem solutions} is complete.

\end{document}